\documentclass[12pt]{article}%
\usepackage{amsmath}%
\usepackage{amsfonts}%
\usepackage{amssymb}%
\usepackage{graphicx}
\usepackage{verbatim}
\providecommand{\U}[1]{\protect\rule{.1in}{.1in}}
\newtheorem{theorem}{Theorem}

\newtheorem{definition}[theorem]{Definition}
\newtheorem{example}[theorem]{Example}

\newtheorem{lemma}[theorem]{Lemma}

\newtheorem{proposition}[theorem]{Proposition}

\newtheorem{question}{Question} \topmargin-2cm
\newenvironment{proof}[1][Proof]{\noindent\textbf{#1.} }{\ \rule{0.5em}{0.5em}}
\begin{document}

\title{Bandlimited Spaces on Some $2$-step Nilpotent Lie Groups With One Parseval Frame Generator}
\author{Vignon Oussa}

\maketitle 
\begin{center}
Saint-Louis University
\end{center}
   \begin{abstract}
Let $N$ be a step two connected and simply connected  non commutative nilpotent Lie group which is square-integrable modulo the center. Let $Z$ be the center of $N$. Assume that $N=P\rtimes M$  such that $P$, and $M$ are simply connected, connected abelian Lie groups, $M$ acts non-trivially on $P$ by automorphisms and $\dim P/Z=\dim M$. We study bandlimited subspaces of $L^2(N)$ which admit Parseval frames generated by discrete translates of a single function. We also find characteristics of bandlimited subspaces of $L^2(N)$ which do not admit a single Parseval frame. We also provide some conditions under which continuous wavelets transforms related to the left regular representation admit discretization, by some discrete set $\Gamma\subset N$. Finally, we show some explicit examples in the last section.
\end{abstract}

\section{Introduction}
In the classical case of $L^2(\mathbb{R})$, closed subspaces where Fourier transforms are supported on a bounded interval enjoy some very nice properties. Such subspaces are called band-limited subspaces of $L^2(\mathbb{R})$. Among other things, these subspaces are stable under the regular representation of the real line; for each class of functions belonging to these spaces there exists an infinitely smooth representative, and more importantly, these spaces admit frames and bases generated by discrete translations of a single function. A classical example is the Paley-Wiener space defined as the space of functions in $L^2(\mathbb{R})$ with Fourier transform supported within the interval  $[-0.5,0.5]$. For such space, the set of integer translates of the sinc function $\frac{\sin(\pi x)}{\pi x}$ forms a Parseval frame, and even better, it is an orthonormal basis for the space (see \cite{chris}). These notions are easily generalized to $L^2(\mathbb{R}^d)$. It is then natural to investigate whether similar results are possible when $\mathbb{R}$ is replaced with a connected, simply connected non commutative Lie group $N$. Since the closest Lie groups to $\mathbb{R}^n$ are simply connected, connected step two nilpotent Lie groups, this class of groups is a natural one to consider. For example, in \cite{than}, Thangavelu has studied Paley Wiener theorems for step two nilpotent Lie group.
In the monograph $\cite {Fuhr cont}$, Hartmut F\"uhr has studied sampling theorems for the Heisenberg group, which is the simplest non commutative nilpotent Lie group of step two. Using various theorems related to Gabor frames, he obtained some nice conditions on how to construct Parseval frames invariant under the left regular representation of the Heisenberg group restricted to some lattice subgroups (chapter 6 in \cite {Fuhr cont}). His results, even though very precise and explicit, were obtained in the restricted case of the Heisenberg Lie group. In this paper, we study subspaces of bounded spectrum of $L^2(N)$ where $N$ belongs to a class of connected, simply connected nilpotent Lie groups satisfying the following conditions. $N$ is a $2$-step nilpotent Lie group which is square-integrable modulo the center. We also assume that $N=P\rtimes M$  such that $P$ and $M$ are simply connected, connected commutative Lie groups such that $P$ is a maximal normal subgroup of $N$ which is commutative, and is containing the center of the group. Furthermore, $M$ acts non-trivially on $P$, and if $\mathrm{Z}$ denotes the center of $N$, then $\dim M = \dim P/Z$. On the Lie algebra level, there exist commutative Lie subalgebras $\mathfrak{m}$, and $\mathfrak{m}_1$ such that $\mathfrak{n}=\mathfrak{m}\oplus\mathfrak{m}_1\oplus\mathfrak{z}$, $\mathfrak{m}$ is the Lie algebra of the subgroup $M$, $\mathfrak{m}_1\oplus\mathfrak{z}$ is the Lie algebra of the maximal normal subgroup $P$, $\dim\mathfrak{m}=\dim\mathfrak{m}_1$, $\mathfrak{z}$ is the center of $\mathfrak{n}$, and finally the adjoint action of $\mathfrak{m}$ on $\mathfrak{n}$ is non-trivial. We answer the following questions.
\begin{question} \label{Q1}
Let $L$ be the left regular representation acting on $L^2(N)$, and let $\mathcal{H}$ be a closed band-limited subspace of $L^2(N)$, how do we pick a discrete subset $\Gamma \subset N$, and a function $\phi$ in $\mathcal{H}$ such that the system $L(\Gamma)\phi$ forms either a Parseval frame or an orthonormal basis in $\mathcal{H}$? 
\end{question}
\begin{question}\label{Q2} What are some necessary conditions for the existence of a single Parseval frame generator for any arbitrary band-limited subspace of $L^2(N)$.\end{question}
\begin{question}\label{Q3} What are some characteristics of band-limited subspaces of $L^2(N)$ which admit discretizable continuous wavelets. What are some characteristics of the quasi-lattices allowing the discretizations?
\end{question}
In order to provide answers to these questions, we relax the definition of lattice subgroups, by considering a broader class of discrete sets which we call quasi-lattices. It turns out that these quasi-lattices must satisfy some specific density conditions which we provide in this paper. We show how to use systems of multivariate Gabor frames to obtain Parseval frames for band-limited subspaces of $L^2(N)$ with bounded multiplicities. 

In the first section, we start the paper by reviewing some background materials. In the second section, we prove our results, and finally we compute some explicit examples in the last section. 
Among several results obtained in this paper, the theorem below is the most important one. 
\begin{theorem}
Let $N$ be a simply connected, connected step two nilpotent Lie group with center $Z$ of the form $N=P\rtimes M$ such that $P$ is a maximal commutative normal subgroup of $N$, where  $M$ is a commutative subgroup, and $\dim(P/Z)=\dim(M)$. Let $\mathcal{H}$ be a multiplicity-free subspace of $L^2(N)$ with bounded spectrum. There exists a quasi-lattice $\Gamma\subset N$ and a function $\phi$ such that the system $\{L(\gamma)\phi:\gamma\in\Gamma\}$ forms a Parseval frame in $\mathcal{H}$.
\end{theorem}

\section{Generalities and notations}

\begin{definition}
Given a countable sequence $\left\{  f_{i}\right\}  _{i\in I}$ of functions in
an separable Hilbert space $\mathcal{H},$ we say $\left\{
f_{i}\right\}  _{i\in I}$ forms a \textbf{frame} if and only if there exist strictly positive real numbers $A,B$ such that for any function $f\in
\mathcal{H}$
\[
A\left\Vert f\right\Vert ^{2}\leq\sum_{i\in I}\left\vert \left\langle
f,f_{i}\right\rangle \right\vert ^{2}\leq B\left\Vert f\right\Vert ^{2}.
\]
In the case where $A=B$, the sequence of functions $\left\{  f_{i}\right\}  _{i\in I}$ forms a \textbf{tight frame}, and if $A=B=1$, $\left\{  f_{i}\right\}  _{i\in I}$ is called
a \textbf{Parseval frame}. Also, if $\left\{  f_{i}\right\}  _{i\in I}$ is a Parseval frame such that for all $i\in I,\left\Vert
f_{i}\right\Vert =1$ then $\left\{  f_{i}\right\}  _{i\in I}$ is an orthonormal basis for $\mathcal{H}$.
\end{definition}

\begin{definition}
A lattice $\Lambda$ in $\mathbb{R}^{2d}$ is a discrete subgroup of the additive group $\mathbb{R}^{2d}$. In other words, $\Lambda=A\mathbb{Z}^{2d}$ for some matrix $A$. We say $\Lambda$ is a full rank lattice if $A$ is nonsingular, and we denote the dual of $\Lambda$ by $\Lambda^{\top}=A^{-1tr}\Lambda$ ($A^{tr}$ denotes the transpose of $A$). We say a lattice is separable
if $\Lambda=A\mathbb{Z}^{d}\times B\mathbb{Z}^{d}.$ A \textbf{fundamental domain} $D$ for a lattice in $\mathbb{R}^{d}$ is a measurable set such that the followings hold
\end{definition}
\begin{enumerate}
\item $(D+\lambda)\cap (D+\lambda^{\prime})\neq \emptyset$ for distinct $\lambda,$ $\lambda^{\prime}$ in $\Lambda.$
\item $\mathbb{R}^{d}={\displaystyle\bigcup\limits_{\lambda\in\Lambda}}\left(  D+\lambda\right).$ We say $D$ is a packing set for $\Lambda$ if
$\sum_{\lambda}\chi_{D}\left(  x-\lambda\right)  \leq1$ for almost every $x\in\mathbb{R}^d.$
\item Let $\Lambda=A\mathbb{Z}^{d}\times B\mathbb{Z}^{d}$ be a full rank lattice in $\mathbb{R}^{2d}$ and $g\in L^{2}\left(\mathbb{R}^{d}\right)$. The family of functions in $L^{2}\left(\mathbb{R}^{d}\right)$, \begin{equation}
\label{Gabor}
\mathcal{G}\left(  g,A\mathbb{Z}^{d}\times B\mathbb{Z}^{d}\right)=\left\{  e^{2\pi i\left\langle k,x\right\rangle }g\left(
x-n\right)  :k\in B\mathbb{Z}^{d},n\in A\mathbb{Z}^{d}\right\}\end{equation} is called a \textbf{Gabor system}.
\end{enumerate}
\begin{definition}
Let $m$ be the Lebesgue measure on $\mathbb{R}^d$, and consider a full rank lattice $\Lambda=A\mathbb{Z}^{d}$ inside $\mathbb{R}^d$.
\begin{enumerate}
\item The \textbf{volume} of $\Lambda$ is defined as $vol\left(  \Lambda\right)  = m\left(\mathbb{R}^{d}/\Lambda\right)  =\left\vert \det A\right\vert .$
\item The \textbf{density} of $\Lambda$ is defined as $d\left(\Lambda\right)=\dfrac{1}{\left\vert \det A\right\vert }.$
\end{enumerate}
\end{definition}

\begin{lemma}(\textbf{Density Condition})\label{density}
Given a separable full rank lattice $\Lambda=A\mathbb{Z}^{d}\times B\mathbb{Z}^{d}$ in $\mathbb{R}^{2d}$. The followings are equivalent
\end{lemma}

\begin{enumerate}
\item There exits $g \in L^2(\mathbb{R}^d)$ such that $\mathcal{G}\left(g, \:A\mathbb{Z}^{d}\times B\mathbb{Z}^{d}\right)$ is a Parseval frame in $L^{2}\left(\mathbb{R}^{d}\right).$

\item $vol\left(\Lambda\right)=\left\vert \det A\det B\right\vert \leq1.$

\item There exists $g\in L^{2}\left(\mathbb{R}^{d}\right)$ such that $\mathcal{G}\left(g, A\mathbb{Z}^{d}\times B\mathbb{Z}^{d}\right)$ is complete in $L^{2}\left(\mathbb{R}^{d}\right)$
\end{enumerate}

\begin{proof}
See Theorem 3.3 in \cite{Han Yang Wang}.
\end{proof}
\begin{lemma}
\label{ONB} Let $\Lambda$ be a full rank lattice in $\mathbb{R}^{2d}$. There exists $g\in L^{2}\left(\mathbb{R}^{d}\right)  $ such that $\mathcal{G}\left(  g,\Lambda\right)$ is an orthonormal basis if and only if $vol\left(\Lambda\right)=1.$ Also, if $\mathcal{G}\left(  g,\Lambda\right)$ is a Parseval frame for $L^2(\mathbb{R}^d)$, then $\|g\|^2 = vol(\Lambda).$
\end{lemma}
\begin{proof}
See \cite{Han Yang Wang}, Theorem 1.3 and Lemma 3.2.
\end{proof}

Next, we start by setting up some notations. We will refer the reader to \cite{ArnalCurrey} for a more thorough exposition on the following discussion. Let $\mathfrak{n}$ be a simply
connected, and connected nilpotent Lie algebra over $\mathbb{R}$ with corresponding Lie group $N=\exp\mathfrak{n}$. Let $\mathfrak{s}$ be a
subalgebra in $\mathfrak{n}$ and let $\lambda$ be a linear functional. We define
the subalgebra $\mathfrak{s}^{\lambda}=\left\{  Z\in\mathfrak{n:}\text{ }\lambda\left[  Z,X\right]
=0\text{ for every }X\in\mathfrak{s}\right\}$
and $
\mathfrak{s}\left( \lambda\right)  =\mathfrak{s}^{\lambda}\cap\mathfrak{s}.$ The ideal
$\mathfrak{z}\left(  \mathfrak{n}\right)  $ denotes the center of the Lie
algebra of $\mathfrak{n,}$ and the coadjoint action on the dual of
$\mathfrak{n}$ is simply the dual of the adjoint action of $\exp\mathfrak{n}$
on $\mathfrak{n}$. Given any $X\in\mathfrak{n}$ the coadjoint action is
defined multiplicatively as follows: $\exp X\cdot \lambda\left(  Y\right)  =\lambda\left(
Ad_{\exp-X}Y\right)$. We fix for $\mathfrak{n}$ a fixed Jordan H\"older basis
$\left\{  Z_{i}\right\}  _{i=1}^{n}$ and we define the subalgebras:
$\mathfrak{n}_{k}=\mathbb{R}$-$\mathrm{span}\left\{  Z_{i}\right\}  _{i=1}^{k}.$ Given any linear
functional $\lambda\in\mathfrak{n}^{\ast},$ we construct the following skew-symmetric matrix:
\[
M\left(  \lambda\right)  =\left[ \lambda\left[  Z_{i},Z_{j}\right]  \right]  _{1\leq
i,j\leq n}.
\]
Notice that $\mathfrak{n}\left( \lambda\right)  =\mathrm{nullspace}\left(  M\left(
\lambda\right)  \right)  .$ Also, for each $\lambda\in$ $\mathfrak{n}^{\ast}$ there is a
corresponding set $\mathbf{e}\left(  \lambda\right)  \subset\left\{  1,2,\cdots
,n\right\}  $ of ``jump indices" defined by
\[
\mathbf{e}\left(  \lambda\right)  =\left\{  1\leq j\leq n:\mathfrak{n}_{k}\text{ not
in }\mathfrak{n}_{k-1}+\mathfrak{n}\left(  \lambda\right)  \right\}  .
\]
For each subset $\mathbf{e}$ inside $\left\{  1,2,\cdots,n\right\}  $ the set
$\Omega_{\mathbf{e}}=\left\{  \lambda\in\mathfrak{n}^{\ast}:\mathbf{e}\left(
\lambda\right)  =\mathbf{e}\right\}  $ is algebraic and $N$-invariant. The union of
all such non-empty layers defines the ``coarse stratification" of $\mathfrak{n}^{\ast}$.
It is known that all coajdoint orbits must have even dimension and there is a
total ordering $\prec$ on the coarse stratification for which the minimal
element is Zariski open and consists of orbits of maximal dimension. Let
$\mathbf{e}$ be the jump indices corresponding to the minimal layer. We
define the following matrix which will be very important for this paper
\begin{equation}
\label{matrixV}
V\left(  \lambda\right)  =\left[  \lambda\left[  Z_{i},Z_{j}\right]  \right]
_{i,j\in\mathbf{e}}.
\end{equation}
From now on, we fix the layer
\begin{equation}
\Omega=\left\{  \lambda\in\mathfrak{n}^{\ast}:\det M_{\mathbf{e}^{\prime}}\left(
\lambda\right)  =0\text{ for all }\mathbf{e}^{\prime}\prec\mathbf{e}\text{ and }\det
M_{\mathbf{e}}\left(  \lambda\right)  \neq0\text{ }\right\}  .\label{omega}%
\end{equation}
We define the polarization subalgebra associated with the linear functional
$\lambda$%
\[
\mathfrak{p}(\lambda)= \Sigma_{k=1}^{n} \left( \mathfrak{n}_{k}\left(  \lambda\right)  \cap\mathfrak{n}_{k}\right).
\]
$\mathfrak{p}(\lambda)$ is a maximal subalgebra subordinated to $\lambda$ such that $\lambda[\mathfrak{p}(\lambda),\mathfrak{p}(\lambda)]=0$ and $\chi_{\lambda}(\exp X)=e^{2\pi i \lambda(X)}$ defines a character on $\exp(\mathfrak{p}(\lambda))$. In general, we have for some positive integer $d\geq 1$
\begin{enumerate}
\item $\dim\left(\mathfrak{n}/\mathfrak{n}(\lambda)\right)=2d$.
\item $\mathfrak{p}(\lambda)$ is an ideal in $\mathfrak{n}$ and $\dim\mathfrak{p}(\lambda)=n-d.$
\item $\dim\left(\mathfrak{n}/\mathfrak{p}(\lambda)\right)=$ $d.$
\end{enumerate}
For each linear functional $\lambda$, let $\mathfrak{a}(\lambda)$ and $\mathfrak{b}(\lambda)$ be subalgebras of $\mathfrak{n}$ such that $\mathfrak{a}(\lambda)$ is isomorphic to $\mathfrak{n}/\mathfrak{p}(\lambda)$ and $\mathfrak{b}(\lambda)$ is isomorphic to $\mathfrak{p}(\lambda)/\mathfrak{n}(\lambda)$. We let 
\begin{align*}
\mathfrak{a}(\lambda) &  =\mathbb{R}\text{ -}\mathrm{span}\text{ }\left\{  X_{i}\left(  \lambda\right)  \right\}
_{i=1}^{d},\\
\mathfrak{b}(\lambda)  &  =\mathbb{R}\text{ -}\mathrm{span}\text{ }\left\{  Y_{i}\left(  \lambda\right)  \right\}
_{i=1}^{d},\\
\mathfrak{n}\left(  \lambda\right)   &  =\mathbb{R}\text{ -}\mathrm{span}\text{ }\left\{  Z_{i}\left(  \lambda\right)  \right\}
_{i=1}^{n-2d},
\end{align*}
and
$
\mathfrak{n}=\mathfrak{n}(\lambda)\oplus\mathfrak{b}(\lambda)\oplus\mathfrak{a}(\lambda) .
$

\begin{lemma}
Given $\lambda\in\Omega,$ if $\mathfrak{n}\left(\lambda\right)$ is a constant subalgebra for any
linear functional $\lambda$ then $\mathfrak{n}\left(\lambda\right)=\mathfrak{z}\left(
\mathfrak{n}\right).$
\end{lemma}
\begin{proof}
First, it is clear from its definition that $\mathfrak{n}\left(  \lambda\right)  \supseteq
\mathfrak{z}\left(  \mathfrak{n}\right)  .$ Second, let us suppose that
there exits some $W\in\mathfrak{n}\left(  \lambda\right)  $ such that $W$ is not a
central element. Thus, there must exist at least one basis element $X$ such
that $\left[  W,X\right]  $ is non-trivial but $\lambda\left[
W,X\right]  =0.$ Using the structure constants of the Lie algebra, let us
supposed that $\left[  W,X\right]  =\sum_{k}c_{k}Z_{k}$ for some non-zero constant real numbers $c_k$. Then it must be the
case that $\sum_{k}c_{k}\lambda_{k}=0$ where $\lambda_k$ is the $k$-th coordinate of the linear function $\lambda$ for all $\lambda\in\Omega$. By the linear independency of the coordinates of $\lambda$, $c_{k}=0$ for all $k$. We reach a contradiction.
\end{proof}
\vskip 0.5cm
According to the orbit method, all irreducible representations of $N$ are in
one-to-one correspondence with coadjoint orbits which are parametrized by a
smooth cross-section $\Sigma$ homeomorphic with $\Omega/N$ via Kirillov map.
Defining for each linear functional $\lambda$ in the generic layer, a character of
$\exp\mathfrak{p}(\lambda)$ such that $\chi_{\lambda}\left(  \exp X\right)  =e^{2\pi
i\lambda\left(  X\right)  }$, we realize almost all the unitary irreducible
representations of $N$ ``a la Mackey" as $\pi_{\lambda}=\mathrm{Ind}_{\exp\mathfrak{p}(\lambda)}^{N}\left(  \chi_{\lambda}\right)  .$ An explicit realization of $\left\{  \pi
_{\lambda}:\lambda\in\Sigma\right\}  $ is discussed later on in this section. We invite
the reader to refer to \cite{ArnalCurrey} for more details concerning the
construction of $\Sigma.$ \newline For the remaining of this paper, we will assume that
we are only dealing with a ``nicer" class of nilpotent Lie algebras such that
the following hold:

\begin{enumerate}
\item For any linear functional $\lambda$ in the layer $\Omega,$ the polarization
subalgebra $\mathfrak{p}\left(\lambda\right)$ is constant, and the stabilizer subalgebra $\mathfrak{n}\left(\lambda\right)$ for the coadjoint action on $N$ on $\lambda\in \Omega$ is constant as well. In other words, there exit bases for $\mathfrak{p}(\lambda)$ and $\mathfrak{n}(\lambda)$ which do not depend on the linear functional $\lambda.$ We simply write $\mathfrak{p}(\lambda)=\mathfrak{p}$ and $\mathfrak{n}(\lambda)=\mathfrak{z(\mathfrak{n})}.$

\item $\mathfrak{n}/\mathfrak{p}$, $\mathfrak{p}$ and $\mathfrak{p}/\mathfrak{z(\mathfrak{n})}$ are commutative
algebras such that 
$$
\mathfrak{n}  =  \mathfrak{z(\mathfrak{n})} \oplus\left(\mathbb{R}Y_{1}\oplus\cdots\oplus
\mathbb{R}Y_{d}\right)  \oplus\left(\mathbb{R}X_{1}\oplus\cdots\oplus\mathbb{R}X_{d}\right) $$ with  
$\mathfrak{p}=\mathfrak{z(\mathfrak{n})} \oplus\left(\mathbb{R}Y_{1}\oplus\cdots\oplus\mathbb{R}Y_{d}\right)
$
and
$
\mathfrak{n}=\mathfrak{p}\oplus\left(\mathbb{R}X_{1}\oplus\cdots\oplus\mathbb{R}X_{d}\right).
$

\item $\mathfrak{n}$ is 2-step. In other words, $\left[  \mathfrak{n,n}%
\right]  \subset\mathfrak{z}\left(  \mathfrak{n}\right)  \ $ and given any
$X_{k,}Y_{r}\in\mathfrak{n},$ $\left[  X_{k},Y_{r}\right]  =\sum_{k_{r_{j}}%
}c_{k_{r_{j}}}Z_{k_{r_{j}}},$ where $c_{k_{r_{j}}}$ are structure constants
which are not necessarily nonzero. Letting $\mathfrak{m}_1=\mathbb{R}Y_1\oplus\cdots\oplus\mathbb{R}Y_d$,  $\mathfrak{m}=\mathbb{R}X_1\oplus\cdots\oplus\mathbb{R}X_d$ and $M =\exp(\mathfrak{m})$. $P=\exp\mathfrak{p}$, and $M_1=\exp\mathfrak{m}_1$ are commutative Lie groups such that $N=P\rtimes M.$ $M$ acts on $P$ as follows. For any $m\in M$ and $x\in P$, $m\cdot x=Ad_{m}x=mxm^{-1}$ and the matrix representing the linear operator $ad\log(m)$ is a nilpotent matrix with $ad\log m\neq\textbf{0}$ but $(ad\log m)^2=\textbf{0}$ ($\textbf{0}$ is the $n\times n$ matrix with zero entries everywhere).
\end{enumerate}
There is a fairly large class of nilpotent Lie groups which satisfy the criteria above. Here are just a few examples.
\begin{enumerate}
\item Let $\mathbb{H}$ be the $2d+1$-dimensional Heisenberg Lie group, with Lie algebra spanned by the basis $\{Z,Y_1,\cdots, Y_d,X_1,\cdots, X_d\}$ with the following non-trivial Lie brackets $[X_i,Y_i]=Z$ for $1\leq i \leq d.$ Now, let $N=\mathbb{H}\times \mathbb{R}^k.$ Both $N$ and $\mathbb{H}$ belong to the class of nilpotent Lie groups described above. 
\item Let $N$ be a nilpotent Lie group, with its Lie algebra $\mathfrak{n}$ spanned by the following basis $\{Z_1,Z_2,Y_1,Y_2,X_1,X_2\}$ with non-trivial Lie brackets $$[X_1,Y_1]=[X_2,Y_2]=Z_1,$$ and $[X_1,Y_2]=[X_2,Y_1]=Z_2$. This group also satisfies all the conditions above.
\item Let $N$ be a nilpotent Lie group with its Lie algebra $\mathfrak{n}$ spanned by the basis $\{Z_1,Z_2,Z_3,Z_4,Y_1,Y_2,X_1,X_2\}$ with the following nontrivial Lie brackets $[X_1,Y_1]=Z_2,[X_2,Y_1]=[X_1,Y_2]=Z_3,$ and $ [X_2,Y_2]=Z_4$. There is a generalization of this group which we describe here. Fix a natural number $d$. Let $N$ be a nilpotent Lie group with Lie algebra $\mathfrak{n}$ spanned by the following basis $\{Z_1,\cdots,Z_{2d},Y_1,\cdots,Y_d,X_1,\cdots,X_d\},$ with the following non-trivial Lie brackets; for $i,j\geq 1$, and $i,j\leq d$, $[X_j,Y_i]=Z_{i+j}$. The center of $\mathfrak{n}$ is $2d$-dimensional and the commutator ideal $[\mathfrak{n},\mathfrak{n}]$ is spanned by $\{Z_2,\cdots, Z_{2d}\}.$
\end{enumerate}

\begin{definition}
For a given basis element $Z_{k}\in\mathfrak{n}$, we define the dual basis element $\lambda_{k}\in\mathfrak{n}^{\ast}$
such that
\[
\lambda_{k}\left(  Z_{j}\right)  =\left\{
\begin{array}
[c]{c}
0\text{ if }k\neq j\\
1\text{ if }k=j
\end{array}
\right.  .
\]
\end{definition}

\begin{lemma}
Under our assumptions, for this class of groups, a cross-section for the coadjoint orbits of $N$ acting on the dual of $\mathfrak{n}$ is described as follows
\begin{align*}
\Sigma &  =\left\{  \left(  \lambda_{1},\cdots,\lambda_{n-2d},0,\cdots,0\right)  \right\}
\cap\Omega =\mathfrak{z}\left(  \mathfrak{n}\right)  ^{\ast}\cap\Omega.
\end{align*}
Furthermore identifying $\mathfrak{z}\left(  \mathfrak{n}\right)  ^{\ast}$ with
$\mathbb{R}^{n-2d},$ $\Sigma$ is a dense and open co-null subset of $\mathbb{R}^{n-2d}$ with respect to  the canonical Lebesgue measure.
\end{lemma}
\begin{proof}
The jump indices for each $\lambda$ being $\mathbf{e=}\left\{  n-2d+1,\cdots
,n\right\}  .$ By Theorem 4.5 in \cite{ArnalCurrey}, $\Sigma=\left\{  \left(
\lambda_{1},\cdots,\lambda_{n-2d},0,\cdots,0\right)  \right\}  \cap \:\Omega.$ Referring to
the definition of $\Omega$ in (\ref{omega}), the proof of the rest of the lemma
follows. Notice that $\mathrm{\det}\left(  V\left(  \lambda\right)  \right)  $
is a non-zero polynomial function defined on $\mathfrak{z}\left(
\mathfrak{n}\right)  ^{\ast}=\mathbb{R}^{n-2d}.$ Thus, $\mathrm{\det}\left(  V\left(  \lambda\right)  \right)  $ is
supported on a co-null set of $\mathbb{R}^{n-2d}$ with respect to the Lebesgue measure.
\end{proof}

We refer the reader to \cite{Corwin} which is a standard reference book for representation theory of nilpotent Lie groups. In this paragraph, we will give an almost complete description of the unitary irreducible representations of $N$. They are almost all parametrized by $\Sigma$ and they are of the form 
$\pi_{\lambda}=\mathrm{Ind}_{\exp\mathfrak{p}(\lambda)}^{N}\left(  \chi_{\lambda}\right)  $ ($\lambda\in\Sigma$) acting in the
Hilbert completion of the functions space
\[
\mathbf{B}=\left\{
\begin{array}
[c]{c}
f:N\rightarrow\mathbb{C}\text{ such that }f\left(  xy\right)  =\chi_{\lambda}\left(  y\right)  ^{-1}f\left(
x\right)  \text{ for }y\in\exp\mathfrak{p,}\text{ }\\
\text{and }x\in N/\exp\mathfrak{p}\text{ and }\int_{N/\exp\mathfrak{p}
}f\left(  x\right)  d\overline{x}<\infty
\end{array}
\right\}
\]
which is isometric and isomorphic with $L^{2}\left(  N/\exp\mathfrak{p}\right)  $ which
we naturally identify with $L^{2}\left(\mathbb{R}^{d}\right)  $ via the identification
\[
\exp\left(  x_{1}X_{1}+\cdots+x_{d}X_{d}\right)  \mapsto\left(  x_{1}
,\cdots,x_{d}\right)  .
\]
The action of $\pi_{\lambda}$ is obtained in the following way: $\pi_{\lambda}\left(
x\right)  f\left(  y\right)  =f\left(  x^{-1}y\right)  $ for $f\in\mathbf{B.}$ We fix a coordinate system for the element of $N$. More precisely, for any $n\in N$, 
\[
n=\exp\left(  z_{1}Z_{1}+\cdots+z_{n-2d}Z_{n-2d}\right)  \exp\left(  y_{1}
Y_{1}+\cdots+y_{d}Y_{d}\right)  \exp\left(  x_{1}X_{1}+\cdots+x_{d}
X_{d}\right)
\]
and we have,
\begin{enumerate}
\item Let $F\in L^2\left(\mathbb{R}^d\right)$, 
$$\pi_{\lambda}\left(  \exp z_{k}Z_{k}\right)  F\left(  x_{1},\cdots,x_{d}\right)
=e^{2\pi i\lambda z_{k}}F\left(  x_{1},\cdots,x_{d}\right)  \text{ for }Z_{k}
\in\mathfrak{z}\left(  \mathfrak{n}\right) .$$ Elements of the center of the group act on $L^2(\mathbb{R}^d)$ by multiplications by characters.
\item  
$\pi_{\lambda}\left(  \exp\left(  t_{1}X_{1}+\cdots+t_{d}X_{d}\right)  \right)
F\left(  x_{1},\cdots,x_{d}\right)  =F\left(  x_{1}-t_{1},\cdots,x_{d}
-t_{d}\right).
$ Thus, elements of the subgroup $M$ act by translations on $L^2(\mathbb{R}^d)$.
\item 
Put $x=\left(  x_{1},\cdots,x_{d}\right)  ,$ $y=\left(  y_{1},\cdots
,y_{d}\right)$ and define for $\lambda\in \Sigma$,
\begin{equation}\label{matrixB}
B\left(\lambda\right)=-
\left(
\begin{array}
[c]{ccc}
\lambda\left[  X_{1},Y_{1}\right]   & \cdots & \lambda\left[  X_{d}%
,Y_{1}\right]  \\
\vdots &  & \vdots\\
\lambda\left[  X_{d},Y_{1}\right]   & \cdots & \lambda\left[  X_{d}%
,Y_{d}\right]
\end{array}
\right).
\end{equation}
$\pi_{\lambda}\left(  \exp y_{1}Y_{1}\cdots\exp y_{d}Y_{d}\right)  F\left(x\right)  =e^{2\pi i\left\langle x^{tr},\: B\left(  \lambda\right)
y^{tr}\right\rangle }F\left(  x\right) .$ Therefore, elements of the subgroup $M_1$ act by modulations on $L^2(\mathbb{R}^d)$.
\end{enumerate}
This completes the description of all the unitary irreducible representations of $N$ which will appear in the Plancherel transform. Next, we consider the Hilbert space $L^{2}\left(  N\right)  $ where $N$ is
endowed with its canonical Haar measure. $\mathcal{P}$ denotes the Plancherel
transform on $L^{2}\left(  N\right)  ,$ $\lambda=\left(  \lambda_{1},\cdots
,\lambda_{n-2d}\right) \in \Sigma $ and $d\mu\left(  \lambda\right)  =\left\vert
\mathrm{\det}\left( B\left(  \lambda\right)  \right)  \right\vert
d\lambda$ is the Plancherel measure (see chapter 4 in \cite{Corwin}). We have $$\mathcal{P}:L^{2}\left(  N\right)
\rightarrow\int_{\Sigma}^{\oplus}L^{2}\left(\mathbb{R}^{d}\right)  \otimes L^{2}\left(\mathbb{R}^{d}\right)  d\mu\left(  \lambda\right)  $$ where the Fourier transform is defined on $L^2(N)\cap L^1(N)$ by 
$$
\mathcal{F}\left(  f\right)\left(\lambda\right)  =\int_{\Sigma}f\left(  n\right)  \pi_{\lambda
}\left(  n\right) dn  ,
$$ and the Plancherel transform is the extension of the Fourier transform to $L^2(N)$ inducing the equality 
$
\left\Vert f\right\Vert _{L^{2}\left(  N\right)  }^{2}=\int_{\Sigma}\left\Vert
\mathcal{P}\left(  f\right)  \left(  \lambda\right)  \right\Vert_{\mathcal{HS}} ^{2}
d\mu\left(  \lambda\right)$  ($||\cdot ||_{\mathcal{HS}}$ denotes the Hilbert Schmidt norm on $L^2\left(\mathbb{R}^{d}\right)  \otimes L^{2}\left(\mathbb{R}^{d}\right)$). Let $L$ be the left regular representation of the group $N.$ We have,
\[
L \simeq \mathcal{P}L\mathcal{P}^{-1}=\int_{\Sigma}^{\oplus}\pi_{\lambda}\otimes\mathbf{1}_{L^{2}\left(\mathbb{R}^{d}\right)  }d\mu\left(  \lambda\right),
\]
where 
$\mathbf{1}_{L^{2}\left(\mathbb{R}^{d}\right)}$ is the identity operator on $L^{2}\left(\mathbb{R}^{d}\right)$ and the following holds almost everywhere: $\mathcal{P}(L(x)\phi)(\lambda)=\pi_{\lambda}(x)\circ \mathcal{P}\phi(\lambda).$ Furthermore the Plancherel transform is used to characterize all left-invariant subspaces of $L^2(N)$.  In fact, referring to Corollary $4.17$ in $\cite{Fuhr cont}$, the projection $P$ onto any left-invariant subspace of $L^2(N)$ corresponds to a field of projections such that  $\mathcal{P}P\mathcal{P}^{-1}\simeq \int_{S}^{\oplus}(\mathbf{1}_{L^2(\mathbb{R}^d)}\otimes \widehat{P}_{\lambda})d\mu(\lambda)$ where $S$ is measurable subset of $\Sigma$, and for $\mu$ a.e. $\lambda,$ $\widehat{P}_{\lambda}$ corresponds to a projection operator onto $L^2(\mathbb{R}^d).$

In general, a lattice subgroup $\Gamma$ is a uniform subgroup of $N$ i.e
$N/\Gamma$ is compact and $\log\Gamma$ is an additive subgroup of
$\mathfrak{n.}$ Since such class of discrete sets is too restrictive, we relax the definition to obtain some \textbf{quasi lattices} in $N$. 
\begin{definition} Let $U$, $W$ be two Borel subsets of $N$, and $\Gamma\subset N$ is countable.  We say that $\Gamma$ is \textbf{$U$-dense} if $\Gamma U=\cup_{\gamma\in\Gamma}\gamma U=G$. $\Gamma$ is called \textbf{$W$-separated} if $\gamma W\cap \gamma'W$ is a null set of $N$ for distinct $\gamma,\gamma'\in \Gamma$, and $\Gamma$ is a \textbf{quasi-lattice} if there exists a relatively compact Borel set $C$ such that $\Gamma$ is both $C$-separated and $C$-dense.
\end{definition}
\begin{definition}
 Let $a,q,b$ be vectors with strictly positive real number entries such that $a=(a_1,\cdots,a_{n-2d})$, $b=(b_1\cdots b_d) $ and $q=(q_1,\cdots q_d)$. We denote $\Gamma_{a,q,b}$ the family of \textbf{quasi lattices} such that 
\[
\Gamma_{a,q,b}=\left\{
\begin{array}
[c]{c}{\displaystyle\prod\limits_{j=1}^{n-2d}}\exp\left(  \dfrac{m_{j}}{a_{j}}Z_{j}\right){\displaystyle\prod\limits_{j=1}^{d}}\exp\left(  \dfrac{k_{j}}{q_{j}}Y_{j}\right){\displaystyle\prod\limits_{j=1}^{d}}\exp\left(  \dfrac{n_{j}}{b_{j}}X_{j}\right) :  \\
m_{j},k_{j},n_{j}\in\mathbb{Z}
\end{array}
\right\}  .
\]
Elements of $\Gamma_{a,q,b}$ will be of the type 
$
\gamma_{a,q,b} =\exp\left(  \frac{m_{1}}{a_1}Z_{1}\right)  \cdots\exp\left(
\frac{m_{n-2d}}{a_{n-2d}}Z_{n-2d}\right)$ $\left(  \exp \frac{k_{1}}{q_{1}}Y_{1}\right)
\cdots$ $\left(  \exp \frac{k_{d}}{q_{d}}Y_{d}\right)$ $  \exp\left(  \frac{n_{1}}{b_1}X_{1}+\cdots+\frac{n_{d}}{b_d}X_{d}\right).$ For each fixed quasi-lattice $\Gamma_{a,q,b}$ we also define the corresponding \textbf{reduced quasi lattice}
\[
\Gamma_{q,b}=\left\{{\displaystyle\prod\limits_{j=1}^{d}}\exp\left(  \frac{k_{j}}{q_j}Y_{j}\right){\displaystyle\prod\limits_{j=1}^{d}}\exp\left(  \frac{n_{j}}{b_j}X_{j}\right)  :k_{j},n_{j}\in\mathbb{Z}\right\}  .
\]
Elements of the reduced quasi lattice will be of the type $$\gamma_{q,b}=\left(  \exp \frac{k_{1}}{q_{1}}Y_{1}\right)
\cdots\left(  \exp \frac{k_{d}}{q_{d}}Y_{d}\right)  \exp\left(  \frac{n_{1}}{b_1}X_{1}%
+\cdots+\frac{n_{d}}{b_d}X_{d}\right)  .$$
\end{definition} 
\begin{definition} We say a function $f\in L^2(N)$ is \textbf{band-limited} if its Plancherel transform is supported on a bounded measurable subset of $\Sigma$.  
\end{definition}

Let $\mathbf{I}\subseteq \{\lambda\in \Sigma : 0\leq \lambda_i \leq a_i\}$ (without loss of generality, one could take $\mathbf{I}\subseteq \{\lambda\in \Sigma : -a_i/2\leq \lambda_i \leq a_i/2\}$). We fix $\left\{  \mathbf{u}\left(
\lambda\right) =\mathbf{u} :\lambda\in\mathbf{I}\right\}$ a measurable field of unit vectors in $L^2\left(\mathbb{R}^{d}\right).$ We consider the multiplicity-free subspace $\mathbf{F=}\int_{\mathbf{I}}^{\oplus}L^{2}\left(\mathbb{R}^{d}\right)  \otimes\mathbf{u}  $ $d\mu\left(\lambda\right) $ which is naturally isomorphic and isometric with $\int_{\mathbf{I}}^{\oplus}L^{2}\left(\mathbb{R}^{d}\right)  d\mu\left(  \lambda\right)  $ via the mapping: $
\left\{  f_{\lambda}\otimes\mathbf{u}  \right\}
_{\lambda\in\mathbf{I}}\mathbf{\mapsto}\text{ }\{f_{\lambda}\}_{\lambda\in\mathbf{I}}.
$ Observe that 
$$\left\{  \prod\limits_{k=1}^{n-2d}\dfrac{e^{2\pi i\left\langle \frac{m_k}{a_k},\cdot\right\rangle }}
{\sqrt{a_k}}:m_k\in\mathbb{Z}\right\}  $$ forms a Parseval frame for $L^{2}\left(  \mathbf{I}\right).$ Next, let $b=(b_1,\cdots,b_d)$, and $q=(q_1,\cdots,q_d).$ We define the $d\times d$ diagonal matrix $D(q)$ with entry $\frac{1}{q_i}$ on the ith row, and similarly, we define the following $d\times d$ matrix
\begin{equation}
A\left(  b\right)  =\left(
\begin{array}
[c]{ccc}
\dfrac{1}{b_1} & \cdots & 0\\
\vdots & \ddots & \vdots\\
0 & \cdots & \dfrac{1}{b_d}
\end{array}
\right)   \label{matrixB}.
\end{equation}
These matrices will be useful for us later.

As a general comment, we would like to mention here that, due to Hartmut F\"uhr, the concept of continuous wavelets associated to the left regular representation of locally compacts type I groups is well understood. A good source of reference is the monograph \cite{Fuhr cont}. We also bring to the reader's attention the following fact. In the case of the Heisenberg group, Azita Mayeli provided in $\cite{Azita}$ an explicit construction of band-limited Shannon wavelet using notions of frame multiresolution analysis. 

\begin{definition}
Let $\left(  \pi,\mathcal{H}_{\pi}\right)  $ be a unitary representation of $N$.  We define the map $\mathcal {W_{\eta}}:\mathcal{H}_{\pi}\rightarrow L^2(N)$ such that $\mathcal {W_{\eta}}\phi(x)=\langle \phi, \pi(x) \eta \rangle$. A vector $\eta\in\mathcal{H}_{\pi}$ is called \textbf{admissible} for the representation $\pi$ if $\mathcal {W_{\eta}}$ defines an isometry on $\mathcal{H}_{\pi}.$ In this case, $\eta$ is called a \textbf{continuous wavelet} or an admissible vector. 
\end{definition}
Let $L$ denote the left regular representation, due to Hartmut F\"uhr \cite {Fuhr cont}, it is known that in general
for a non discrete locally compact topological group of type I, $\left(  L,L^{2}\left(
G\right)  \right)  $ is admissible if and only if $G$ is nonunimodular. Thus,
in fact for our class of groups, $\left(  L,L^{2}\left(  N\right)  \right)  $
is not admissible since any nilpotent Lie group is unimodular. However, there are subspaces of $L^{2}\left(  N\right)  $
which admit continuous wavelets for $L.$

\begin{lemma}\label{cont}
Given the closed left-invariant subspace of $L^2\left(  N\right)  ,$
$\mathcal{H=P}^{-1}\left(  \mathbf{H}\right)  ,$ such that
\[
\mathbf{H}=\int_{\mathbf{I}}^{\oplus}L^{2}\left(\mathbb{R}^{d}\right)  \otimes\mathbb{C}\text{-span}\{\mathbf{u}_1\left(  \lambda\right),\cdots,\mathbf{u_{m(\lambda)}}\left(  \lambda\right)\}  \text{ }d\mu\left(
\lambda\right)  .
\]
Assuming that $ \{\mathbf{u}_1\left(  \lambda\right),\cdots,\mathbf{u_{m(\lambda)}}\left(  \lambda\right)\}$ is an orthonormal set and $\left(  L|\mathcal{H},\mathcal{H}\right)  $ is admissible, an admissible vector $\eta$ satisfies the following criteria:
 $\left\Vert \eta\right\Vert ^{2}=\int_{\mathbf{I}}\mathbf{m}(\lambda)d\mu\left(
\lambda\right).$
\end{lemma}
\begin{proof}
See Theorem 4.22 in \cite{Fuhr cont}.
\end{proof}

\section{Results}
In this section, we will provide solutions to the problems mentionned in Question \ref{Q1}, Question \ref{Q2}, and Question \ref{Q3} in the introduction of the paper. We start by fixing some notations which will be used throughout this section. Let $\mathcal{H}=\mathcal{P}^{-1}(\mathbf{F})$ be a multiplicity-free subspace of $L^2(N)$ such that $$\mathbf{F=}\int_{\mathbf{I}}^{\oplus}L^{2}\left(\mathbb{R}^{d}\right)  \otimes\mathbf{u}\: d\mu\left(\lambda\right) ,$$ and $\mathbf{u}$ is a fixed unit vector in $L^2(\mathbb{R}^d).$ Recall that $b=(b_1,\cdots,b_d)$ and $q=(q_1,\cdots,q_d)$. 

\begin{lemma}\label{gaborsystem}
 Let $\phi\in\mathcal{H}$ such that $\mathcal{P}(\phi)(\lambda)=F(\lambda)\otimes\mathbf{u}$ a.e. Recall the matrix $B(\lambda)$ as defined in (\ref{matrixB}). For almost every linear functional $\lambda\in\mathbf{I},$  $F\left(  \lambda
\right)  \in L^{2}\left(\mathbb{R}^{d}\right)$, and $\left\{  \pi_{\lambda}\left( \gamma_{q,b}\right)  F\left(
\lambda\right)  \right\}  _{\gamma_{q,b}}$ forms a
multivariate Gabor system (\ref{Gabor}) of the type $\mathcal{G}\left(  F\left(
\lambda\right)  ,\Lambda\left(  \lambda\right)  \right)  $ such that
$\Lambda\left(  \lambda\right)  $ is a separable full rank lattice of the
form $\Lambda\left(  \lambda\right)  =A\left(  b\right) \mathbb{Z}^{d}\times B\left(  \lambda\right) D(q)\mathbb{Z}^{d}.$ Furthermore, for a.e. $\lambda\in\mathbf{I}$, $$\mathrm{Vol}(\Lambda(\lambda))=\dfrac{\vert \det B(\lambda)\vert}{b_1\cdots b_{n-2d}q_1\cdots q_{n-2d}}.$$
\end{lemma}
\begin{proof}
Following our description of the irreducible representations of $N$, we simply compute the action of the unitary irreducible representations restricted to the reduced quasi-lattice $\Gamma_{q,b}$. Given $F(\lambda)\in L^2(\mathbb{R}^d)$, and $\gamma_{q,b}\in \Gamma_{q,b}$, some simple computations show that
$$
\pi_{\lambda}\left(  \gamma_{q,b}\right)  F(\lambda)\left(  x_{1},\cdots,x_{d}\right)
 =e^{2\pi i\left\langle x^{tr},\: B\left(  \lambda\right)D(q)  k^{tr}\right\rangle
}F(\lambda)\left(  x_{1}-\frac{n_{1}}{b_1},\cdots,x_{d}-\frac{n_{d}}{b_d}\right)  .
$$
\end{proof}

\begin{proposition}
\label{main prop} Let $\phi$ be a vector in $\mathcal{H}.$  If $\left\{  L\left(  \gamma_{a,q,b}\right)  \phi\right\}  _{\gamma_{a,q,b}\in\Gamma_{a,q,b}}$
is a Parseval frame, then for $\mu$ a.e. $\lambda\in\mathbf{I},$ the
following must hold:
\begin{enumerate}
\item $\left\{ \prod_{k=1}^{n-2d}\sqrt{a_k}\left\vert \det B(\lambda)\right\vert^{1/2} \pi_{\lambda}\left( \gamma_{q,b} \right) \widehat{\phi}(\lambda ): \gamma_{q,b} \in \Gamma_{q,b} \right\}$ forms a Parseval frame in $L^2(\mathbb{R}^d)\otimes\mathbf{u}\simeq L^2(\mathbb{R}^d).$
\item $\mathrm{Vol}(\Lambda(\lambda))=\det\left\vert A\left(  b\right) B\left(  \lambda\right) D(q) \right\vert
\leq1.$
\end{enumerate}
\end{proposition}
\begin{proof}
Given any function $\psi\in\mathcal{P}^{-1}\left(  \mathbf{F}\right)  ,$ we
have $\sum_{\gamma_{a,q,b}}\left\vert \left\langle \psi,L\left(  \gamma
_{a,q,b}\right)  \phi\right\rangle \right\vert ^{2}=\left\Vert \psi\right\Vert
_{L^{2}\left(  N\right)  }^{2}.$ We use the operator $\symbol{94}$ instead of
$\mathcal{P}$ and we define $\widehat{L}=\mathcal{P}L\mathcal{P}^{-1}.$ 
\begin{align}
\label{frame}
\sum_{\gamma_{a,q,b}}\left\vert \left\langle \psi,L\left(  \gamma
_{a,q,b}\right)  \phi\right\rangle_{L^2(N)} \right\vert ^{2}  &  =\sum_{\gamma_{a,q,b}}\left\vert \int_{\mathbf{I}}\left\langle \widehat{\psi}\left(
\lambda\right)  ,\widehat{L}\left(  \gamma_{a,q,b}\right)  \widehat{\phi
}\left(  \lambda\right)  \right\rangle_{\mathcal{HS}} d\mu\left(  \lambda\right)  \right\vert
^{2}\\
&  =\sum_{\gamma_{a,q,b}}\left\vert \int_{\mathbf{I}}\left\langle
\widehat{\psi}\left(  \lambda\right)  ,\pi_{\lambda}\left(  \gamma
_{a,q,b}\right)  \widehat{\phi}\left(  \lambda\right)  \right\rangle_{\mathcal{HS}}
d\mu\left(  \lambda\right)  \right\vert ^{2}.
\end{align}
Using the fact that in $L^{2}\left(  \mathbf{I}\right)  ,$ $$\left\{  \prod_{k=1}^{n-2d}\frac
{e^{2\pi i\left\langle m_k,\lambda_k\right\rangle }}{\sqrt{a_k}}:m_k\in\mathbb{Z},(\lambda_1,\cdots,\lambda_{n-2d},0,\cdots,0)\in\mathbf{I}\right\}  $$ forms a Parseval frame in $L^2(\textbf{I})$, we let $r\left(
\lambda\right)  =\left\vert \mathrm{\det}\left( B\left(
\lambda\right)  \right)  \right\vert,$ and put $$c_{\gamma_{q,b}}(\lambda)=\left(\prod_{k}^{n-2d}\sqrt{a_k}\right)\left\langle
\widehat{\psi}\left(  \lambda\right)  ,\pi_{\lambda}\left(  \gamma_{q,b}\right)
\widehat{\phi}\left(  \lambda\right)  \right\rangle_{\mathcal{HS}} r\left(  \lambda\right).$$
Equation (\ref{frame}) becomes,
\begin{align*}
\sum_{\gamma_{a,q,b}}\left\vert \left\langle \psi,L\left(  \gamma_{a,q,b}\right)
\phi\right\rangle_{L^2(N)} \right\vert ^{2}  &  =\sum_{\gamma_{q,b}}\sum_{m\in\mathbb{Z}^{d}}\left\vert \int_{\mathbf{I}}{\displaystyle\prod\limits_{k=1}^{n-2d}}e^{2\pi i\lambda_k\frac{m_{k}}{a_k}}\left\langle \widehat{\psi}\left(\lambda\right)  ,\pi_{\lambda}\left(  \gamma_{q,b}\right)  \widehat{\phi}\left(\lambda\right)  \right\rangle_{\mathcal{HS}} d\mu\left(  \lambda\right)  \right\vert ^{2}\\
&  =\sum_{\gamma_{q,b}}\sum_{m\in\mathbb{Z}^{d}}\left\vert \int_{\mathbf{I}}{\displaystyle\prod\limits_{k=1}^{n-2d}}\frac{e^{2\pi i\lambda_k\frac{m_{k}}{a_k}}}{\sqrt{a_k}}c_{\gamma_{q,b}}(\lambda)
d\lambda\right\vert ^{2}.\\
\end{align*} Since $c_{\gamma_{q,b}}$ is an element of $L^2\left(\mathbf{I}\right)$, and because $\left\{  \prod_{k=1}^{n-2d}\frac
{e^{2\pi i\left\langle m_k,\cdot\right\rangle }}{\sqrt{a_k}}:m_k\in\mathbb{Z}\right\}  $ forms a Parseval frame,
\begin{equation} \label{last}\sum_{\gamma_{a,q,b}}\left\vert \left\langle \psi,L\left(  \gamma_{a,q,b}\right)
\phi\right\rangle_{L^2(N)}\right\vert ^{2} = \sum_{\gamma_{q,b}} \|c_{\gamma_{q,b} }\|^2.\end{equation}
Next, put $\mathbf{a}=\prod_{k=1}^{n-2d}\sqrt{a_k}.$ Then (\ref{last}) yields
\begin{align*}
\sum_{\gamma_{a,q,b}}\left\vert \left\langle \psi,L\left(  \gamma_{a,q,b}\right)
\phi\right\rangle_{L^2(N)} \right\vert ^{2}&  =\sum_{\gamma_{q,b}}\int_{\mathbf{I}}\left\vert \mathbf{a}\left\langle
\widehat{\psi}\left(  \lambda\right)  ,\pi_{\lambda}\left(  \gamma_{q,b}\right)
\widehat{\phi}\left(  \lambda\right)  \right\rangle_{\mathcal{HS}} r\left(  \lambda\right)  \right\vert ^{2}d\lambda\\
&  =\int_{\mathbf{I}}\sum_{\gamma_{q,b}}\left\vert  \mathbf{a}\left\langle \widehat{\psi}\left(  \lambda\right)  ,\pi_{\lambda}\left(
\gamma_{q,b}\right)  \widehat{\phi}\left(  \lambda\right)  \right\rangle_{\mathcal{HS}}
\sqrt{r\left(  \lambda\right)  }\right\vert ^{2}r\left(  \lambda\right)
d\lambda\\
&  =\int_{\mathbf{I}}\sum_{\gamma_{q,b}}\left\vert \mathbf{a}\left\langle \widehat{\psi}\left(  \lambda\right)  ,\pi_{\lambda}\left(
\gamma_{q,b}\right)  \widehat{\phi}\left(  \lambda\right)  \right\rangle_{\mathcal{HS}}
\left\vert \mathrm{\det}\left( B\left(  \lambda\right)  \right)
\right\vert ^{1/2}\right\vert ^{2}d\mu\left(  \lambda\right)  .
\end{align*}
Due to the assumption that $L\left(\Gamma_{a,q,b}\right)\phi$ is a Parseval frame, we also have $$ \sum_{\gamma_{a,q,b}}\left\vert \left\langle \psi,L\left(  \gamma_{a,q,b}\right)
\phi\right\rangle_{L^2(N)} \right\vert ^{2}=\int_{\mathbf{I}}\left\Vert \widehat{\psi}\left(  \lambda\right)
\right\Vert_{\mathcal{HS}} ^{2}d\mu\left(  \lambda\right).$$
Thus, 
\begin{eqnarray*}\int_{\mathbf{I}}\left(\sum_{\gamma_{q,b}}\left\vert \mathbf{a}\left\langle \widehat{\psi}\left(  \lambda\right)  ,\pi_{\lambda}\left(
\gamma_{q,b}\right)  \widehat{\phi}\left(  \lambda\right)  \right\rangle_{\mathcal{HS}}
\left\vert \mathrm{\det}\left(  B\left(  \lambda\right)  \right)
\right\vert ^{1/2}\right\vert ^{2}-\left\Vert \widehat{\psi}\left(  \lambda\right)
\right\Vert_{\mathcal{HS}} ^{2}\right)d\mu\left(  \lambda\right) =0.\end{eqnarray*}
So, for $\mu$-a.e., $\lambda\in\mathbf{I},$
\begin{equation}
\sum_{\gamma_{q,b}}\left\vert \left\langle \widehat{\psi}\left(
\lambda\right)  ,\mathbf{a}\left\vert \mathrm{\det}\left(  B\left(
\lambda\right)  \right)  \right\vert ^{1/2}\pi_{\lambda}\left(\gamma_{q,b}\right)  \widehat{\phi}\left(  \lambda\right)  \right\rangle_{\mathcal{HS}} \right\vert
^{2}=\left\Vert \widehat{\psi}\left(  \lambda\right)  \right\Vert_{\mathcal{HS}} ^{2}. \label{final} \end{equation} 
However, we want to make sure that equality \ref{final} holds for all functions in  a dense subset of $\mathcal{H}$. For that purpose, we pick a countable dense set $Q\subset\mathcal{H}$ such that the set $\{\widehat{f}(\lambda):f\in Q\}$ is dense in $L^2(\mathbb{R}^d)\otimes\mathbf{u}$ for almost every $\lambda\in \mathbf{I}.$ For each $f\in Q$, equality \ref{final} holds on $\mathbf{I}-N_f$ where $N_f$ is a null set dependent on the function $f$. Thus, for all functions in $Q$ equality \ref{final} is true for all $\lambda\in\mathbf{I}-\bigcup_{f\in Q}\left( N_f\right)$. Finally, the map $$\widehat{\psi}(\lambda)\mapsto \left\langle \widehat{\psi}\left(  \lambda\right)  ,\pi_{\lambda}\left(
\gamma_{q,b}\right) \left\vert \mathrm{\det}\left( B\left(  \lambda\right)  \right)\right\vert^{1/2}\sqrt{a_1 \cdots a_{n-2d}}\:  \widehat{\phi}\left(  \lambda\right)  \right\rangle_{\mathcal{HS}}
$$
defines an isometry on a dense subset of $L^2(\mathbb{R}^d)\otimes\mathbf{u}$ almost everywhere, completing the first part of the proposition. 
 Next, the second part of the proposition is simply true by the density condition of Gabor systems yielding to Parseval frames. See Lemma 3.2 in \cite{Han Yang Wang}.
\end{proof}

\begin{lemma}
For any fixed $\lambda\in\Sigma,$ for our class of groups, $\left\vert \det B\left(
\lambda\right)  \right\vert =\left(\det V\left(  \lambda\right) \right)^{1/2}.$
\end{lemma}
\begin{proof}
For a fix $\lambda\in\Sigma,$ we recall the definition of the corresponding matrix $V(\lambda)$ given in (\ref{matrixV}). Some simple computations show that 
\[
V\left(  \lambda\right) =\left(
\begin{array}
[c]{cc}
\mathbf{0} & B\left(  \lambda\right)  \\
-B\left(  \lambda\right)   & \mathbf{0}
\end{array}
\right)  .
\]
$\det V\left(  \lambda\right)  =\det B\left(  \lambda\right)  ^{2}$ which is non-zero
since $V\left(  \lambda\right)  $ is a non singular matrix of rank $2d.$ It
follows that $\left\vert \det B\left(  \lambda\right)  \right\vert =\left(\det V\left(  \lambda\right) \right)^{1/2}.$
\end{proof}

Now, we are in good position to start making progress toward the answer of the first question.  
\begin{definition}
Let $r\left(
\lambda\right)  =\left\vert \det B\left(  \lambda\right)  \right\vert $ and $a=(a_1,\cdots,a_{n-2d})$, we define
\begin{equation}
\mathbf{s}=\sup_{\lambda\in
\mathbf{I}}\left\{  r\left(  \lambda\right) \right\}. \label{superieur}
\end{equation}
Notice that $\mathbf{s}$ is always defined since $\mathbf{I}$ is bounded.
\end{definition}

\begin{lemma} \label{g1} For $\mu$ a.e. $\lambda \in \mathbf{I}$, there exits some $b=(b_1,\cdots,b_d)$ and $q=(q_1,\cdots,q_d)$ such that $\mathrm{vol}\left(  A(b)\mathbb{Z}^{d}\times B\left(  \lambda\right)D(q)\mathbb{Z}^{d}\right)\leq1$.
\end{lemma}
\begin{proof}
It suffices to pick $b=b(\mathbf{s})=\left(\mathbf{s}^{1/d},\cdots ,\mathbf{s}^{1/d}\right) $ and $q=(q_1,...,q_d)$ such that $$\dfrac{1}{q_1\cdots q_d}\leq 1.$$
\end{proof}

\begin{lemma}
For $\mu$-a.e $\lambda\in\mathbf{I},$ if $q$ is chosen such that $\frac{1}{q_1\cdots q_d}\leq 1,$
 there exists $g\left(\lambda\right)  \in$
$L^{2}\left(\mathbb{R}^{d}\right)$ such that the Gabor system $\mathcal{G}\left(  g\left(\lambda\right)  ,A\left( b(\mathbf{s})\right) \mathbb{Z}^{d}\times B\left(  \lambda\right)D(q)\mathbb{Z}^{d}\right) $ forms an Parseval frame. Furthermore,
$$
\left\Vert g\left(  \lambda\right)  \right\Vert ^{2}=\left\vert \det A\left(
b(\mathbf{s})\right)  \det B\left(  \lambda\right)  \det D(q)\right\vert .
$$
\end{lemma}
\begin{proof}
By Theorem 3.3 in \cite{Han Yang Wang} and Lemma \ref{g1}, the density
condition stated also in Lemma \ref{density} implies the existence of the function $g\left(  \lambda\right)  $
for $\mu$-a.e. $\lambda\in\mathbf{I}$.
\end{proof}

\begin{lemma}
Let $\mathbf{u}$ be a unit norm vector in $L^{2}\left(\mathbb{R}^{d}\right).$ If there exists some vector $\eta$ such that $\left\{  L\left(\gamma_{a,q,b}\right)  \eta\right\}_{\gamma_{a,q,b}\in\Gamma_{a,q,b}}$ forms a Parseval frame in
$\mathcal{H}=\mathcal{P}^{-1}\left(  \int_{\mathbf{I}}^{\oplus}\left(
L^{2}\left(\mathbb{R}^{d}\right)  \otimes\mathbf{u}\right)  d\mu\left(  \lambda\right)  \right)  $ then
$\mu\left(  \mathbf{I}\right)  \leq\left(  q_{1}\cdots q_{d}\right)  \left(
b_{1}\cdots b_{d}\right) \left(  a_{1}\cdots a_{n-2d}\right) .$
\end{lemma}
\begin{proof}
Put $\mathbf{a}={\displaystyle\prod\limits_{k=1}^{n-2d}}\sqrt{a_{k}}.$ Under the assumptions that there exists some quasi-lattice
$\Gamma_{a,q,b}$ and some function $\eta$ such that $\left\{  L\left(
\gamma_{a,q,b}\right)  \eta\right\}  _{\gamma_{a,q,b}\in\Gamma_{a,q,b}}$ forms
a Parseval frame, $\sqrt{\det B\left(  \lambda\right)  }\mathbf{a}$
$\left(  \mathcal{P}\eta\right)  \left(  \lambda\right)  $ forms a Parseval frame in $L^{2}\left(\mathbb{R}^{d}\right)\otimes\mathbf{u}  $ for $\mu$-a.e $\lambda\in\mathbf{I.}$ Thus,
\[
\left\Vert \left(  \mathcal{P}\eta\right)  \left(  \lambda\right)  \right\Vert_{\mathcal{HS}}
^{2}=\frac{1}{\left(  q_{1}\cdots q_{d}\right)  \left(  b_{1}\cdots
b_{d}\right)  \mathbf{a}^{2}}.
\]
Computing the norm of the vector $\eta$, we obtain
\begin{align*}
\left\Vert \eta\right\Vert_{L^2(N)} ^{2}  & =\int_{\mathbf{I}}\left\Vert \left(
\mathcal{P}\eta\right)  \left(  \lambda\right)  \right\Vert _{\mathcal{HS}}^{2}d\mu\left(
\lambda\right)  \\
& =\int_{\mathbf{I}}\frac{1}{\left(  q_{1}\cdots q_{d}\right)  \left(
b_{1}\cdots b_{d}\right)  \mathbf{a}^{2}}d\mu\left(  \lambda\right)  \\
& =\frac{\mu\left(  \mathbf{I}\right)  }{\left(  q_{1}\cdots q_{d}\right)
\left(  b_{1}\cdots b_{d}\right)\left(  a_{1}\cdots a_{n-2d}\right)
}.
\end{align*}
$L$ being a unitary representation, $\left\{  L\left(  \gamma_{a,q,b}\right)
\eta\right\}  _{\gamma_{a,q,b}\in\Gamma_{a,q,b}}$ is a Parseval frame. Thus, $\left\Vert \eta\right\Vert ^{2}\leq1$ and $
\mu\left(  \mathbf{I}\right)  \leq\left(  q_{1}\cdots q_{d}\right)  \left(
b_{1}\cdots b_{d}\right) \left(  a_{1}\cdots a_{n-2d}\right).$ \end{proof}

\begin{proposition}
\label{NTF prop}Let $\mathcal{H}$ be a closed left-invariant subspace of
$L^{2}\left(  N\right)  $ such that $\mathcal{H}=\mathcal{P}^{-1}\left(
\mathbf{F}\right) $ where $\mathbf{F}=\int_{\mathbf{I}}^{\oplus}L^2(\mathbb{R}^d)\otimes\mathbf{u}\:d\mu(\lambda)$. Let $\eta\in\mathcal{H}$ such that 
\begin{equation}\label{pf}\widehat{\eta}\left(  \lambda
\right)  =\frac{g\left(  \lambda\right)\otimes\mathbf{u}  }{\prod_{k=1}^{n-2d}\sqrt{a_k}\sqrt{\det |B\left(
\lambda\right)|}} 
\end{equation}
 and the gabor system $\mathcal{G}\left(  g\left(
\lambda\right)  ,A\left( b(\mathbf{s})\right)\mathbb{Z}^{d}\times B\left(  \lambda\right)D(q)\mathbb{Z}^{d}\right) $ forms an Parseval frame for $\mu$ a.e. $\lambda \in$ $\mathbf{I}.$ The following must hold
\begin{enumerate}
\item  $\left\{  L\left(  \gamma_{a,q,b(\mathbf{s})}\right)  \eta\right\}
_{\gamma_{a,q,b(\mathbf{s})}}$ is a Parseval frame in $\mathcal{H}$. 
\item $\left\{  L\left(  \gamma_{a,q,b(\mathbf{s})}\right)  \eta\right\}
_{\gamma_{a,q,b(\mathbf{s})}}$ is an ONB in $\mathcal{H}$ if \begin{equation} \label{basis} \mu(\mathbf{I})=\frac{\prod_{k=1}^{n-2d}(a_k)}{|\det D(q)\det A(b(\mathbf{s}))|}.\end{equation}
\end{enumerate}
\end{proposition}
\begin{proof}
For part 1, since the density condition can be easily met for some appropriate choice of $q$, the existence of the function $g(\lambda)$ generating the Gabor system is guaranteed by Lemma \ref{density}. Assume that $\eta$ is picked as defined in (\ref{pf}). Let $\mathbf{a}=\prod_{k=1}^{n-2d}\sqrt{a_k}.$ $$\sum_{\gamma_{a,q,b(\mathbf{s})}\in\Gamma}\left\vert \left\langle
\psi,L\left(  \gamma_{a,q,b(\mathbf{s})}\right)  \eta\right\rangle_{L^2(N)}\right\vert ^{2} $$  
\begin{align*}
 &  =\int_{\mathbf{I}}\sum_{ \gamma_{q,b(\mathbf{s})}}\left\vert \left\langle \widehat{\psi}\left(  \lambda\right)  ,\pi_{\lambda
}\left(   \gamma_{q,b(\mathbf{s})}\right)  \mathbf{a}\left\vert \mathrm{\det}\left(
V\left(  \lambda\right)  \right)  \right\vert ^{1/4}\widehat{\eta}\left(
\lambda\right)  \right\rangle_{\mathcal{HS}} \right\vert ^{2}d\mu\left(  \lambda\right) \\
&  =\int_{\mathbf{I}}\sum_{ \gamma_{q,b(\mathbf{s})}}\left\vert \left\langle
\widehat{\psi}\left(  \lambda\right)  ,\frac{\mathbf{a}\left\vert \mathrm{\det
}\left(  B(\lambda)\right)  \right\vert ^{1/2}\pi_{\lambda
}\left(\gamma_{q,b(\mathbf{s})}\right)  g\left(  \lambda\right)\otimes\mathbf{u}  }{\sqrt{|\det B\left(
\lambda\right) | }\mathbf{a}}\right\rangle_{\mathcal{HS}} \right\vert ^{2}d\mu\left(
\lambda\right) \\
&  =\int_{\mathbf{I}}\sum_{ \gamma_{q,b(\mathbf{s})}}\left\vert \left\langle
\widehat{\psi}\left(  \lambda\right)  ,\pi_{\lambda}\left( \gamma_{q,b(\mathbf{s})}\right)
g\left(  \lambda\right)\otimes\mathbf{u}   \right\rangle_{\mathcal{HS}} \right\vert ^{2}d\mu\left(
\lambda\right) \\
&  =\int_{\mathbf{I}}\left\Vert \widehat{\psi}\left(  \lambda\right)
\right\Vert_{\mathcal{HS}} ^{2}d\mu\left(  \lambda\right) \\
&  =\left\Vert \psi\right\Vert_{L^2(N)} ^{2}.
\end{align*} 
In order to prove the second part 2, it suffices to check that $\left\Vert \eta\right\Vert ^{2}=1$ using the fact that
if $\mathcal{G}\left(  g\left(  \lambda\right)  ,A\left(  b\right)  \text{ }\mathbb{Z}^{d}\times B\left(  \lambda\right)D(q)\mathbb{Z}^{d}\right)$ is a Parseval frame in $L^{2}\left(\mathbb{R}^{d}\right)  $ then $\left\Vert g\left(  \lambda\right)  \right\Vert
^{2}=\left\vert \det B\left(  \lambda\right)  \det(D(q))\det A\left(  b(\mathbf{s})\right)
\right\vert .$ Finally combining the fact that $L$ is unitary and that the generator of the Parseval frame is a vector of norm $1$, we obtain ($\ref{basis} $).
\end{proof}

All of lemmas above and propositions above lead to the following theorem.
\begin{theorem}
Given $\mathcal{H}=\mathcal{P}^{-1}(\mathbf{F})$ a closed band-limited multiplicity-free left-invariant subspace of
$L^{2}\left(  N\right)$. There exits a quasi-lattice $\Gamma\subset N$
and a function $f\in\mathcal{H}$ such that $L\left(  \Gamma\right)  f$ forms a
Parseval frame in $\mathcal{H}.$ 
\end{theorem}
The second question is concerned with finding some necessary conditions for the existence of a single Parseval frame generator for any arbitrary band-limited subspace of $L^2(N)$. For such purpose, we will now consider all of the left-invariant closed subspaces of $L^2(N)$. Let $\mathcal{K}$ be a left-invariant closed subspace of $L^2(N)$. A complete characterization of left-invariant closed subspaces of $L^2(G)$ where $G$ is a locally compact type I group is well-known and available in the literature. Referring to corollary $4.17$ in the monograph $\cite{Fuhr cont}$,  $\mathcal{P}\left(  \mathcal{K}\right)
=\int_{\mathbf{\Sigma}}^{\oplus} L^{2}\left(\mathbb{R}^{d}\right)  \otimes P_{\lambda}\left( L^{2}\left(\mathbb{R}^{d}\right)  \right) d\mu\left(  \lambda\right)$, where $P_{\lambda}$ is a measurable field of projections onto $L^2(\mathbb{R}^d)$.  We define the
multiplicity function by $m:\Sigma \rightarrow \mathbb{N}\cup\{0,\infty\}$ and $m\left(  \lambda\right)  =\mathrm{rank}\left(  P_{\lambda
}\right).$ We observe that there is a natural isometric isomorphism between
$\mathcal{P}\left(  \mathcal{K}\right)  $ and $\int_{\Sigma}^{\oplus}%
L^{2}\left(\mathbb{R}^{d}\right)  \otimes\mathbb{C}^{m\left(  \lambda\right)  }d\mu\left(  \lambda\right).$ 

\begin{proposition}\label{multspace}
If there exits some function $\phi\in\mathcal{K}$ such that $\left\{  L\left(
\gamma_{a,q,b}\right)  \phi\right\}  _{\gamma_{a,q,b}}$ forms  an Parseval frame, then for almost $\lambda\in\mathbf{I}$, $|\det B(\lambda)m(\lambda)|\leq \prod_{i=1}^d (b_i q_i).$
\end{proposition}
\begin{proof}
Recall that $$\mathbf{a}={\displaystyle\prod\limits_{k=1}^{n-2d}}\sqrt{a_{k}}.$$ By assumption, given any function $f\in\mathcal{H}$, $\sum_{\gamma_{a,q,b}%
}\left\vert \left\langle f,L\left(  \gamma_{a,q,b}\right)  \phi\right\rangle
\right\vert ^{2}=\left\Vert f\right\Vert ^{2}.$  We have $\widehat{f}(\lambda)=\Sigma_{k=1}^{m(\lambda) }u_{f}^k(\lambda)\otimes e^k(\lambda)$ and similarly, $\widehat{\phi}(\lambda)=\Sigma_{k=1}^{m(\lambda) }u_{\phi}^k(\lambda)\otimes e^k(\lambda)$ such that $u_{f}^k(\lambda),u_{\phi}^k(\lambda)$, $e^k(\lambda)$ $\in\L^2(\mathbb{R}^d)$, and $||e^k(\lambda)||=1$ for a.e. $\lambda\in\mathbf{I}$. Next, we identify $L^2(\mathbb{R}^d)\otimes\mathbb{C}^{m(\lambda)}$ with $\bigoplus_{k=1}^{m(\lambda)} L^2(\mathbb{R}^d)$ in a natural way almost everywhere. For example under such identification, $\Sigma_{k=1}^{m(\lambda) }u_{f}^k(\lambda)\otimes e^k(\lambda)$ is identified with $(u_{f}^1,\cdots, u_{f}^{m(\lambda)})$. Thus, a.e. by following similar steps as seen in the proof of Proposition \ref{main prop}, the system 
$$
\left\{  \mathbf{a}\sqrt{\left\vert \det B\left(  \lambda\right)  \right\vert
}\pi_{\lambda}\left(  \gamma_{q,b}\right)  \widehat{\phi}\left(
\lambda\right)  \right\}  _{\gamma_{q,b}}
$$
forms a Parseval vector-valued Gabor frame also called Parseval superframe for almost every $\lambda\in\mathbf{I}$ in $L^2(\mathbb{R}^d)\otimes\mathbb{C}^{m(\lambda)} $. Since we have a measurable field of Gabor systems,
using the density theorem of super-frames (Proposition 2.6. \cite{Grog}), up to a set of measure zero, we
have $\left\vert \det B\left(  \lambda\right)  \det A\left(  b\right)  \det
D\left(  q\right)  \right\vert \leq\frac{1}{m\left(  \lambda\right)  }$ and, $|\det B(\lambda)m(\lambda)|\leq \prod_{i=1}^d (b_i q_i)$.\end{proof}

The following proposition gives some conditions which allow us to provide some answers to Question $3$.
\begin{proposition}\label{corr}
Let $\mathcal{K}$ be a band-limited subspace of $L^{2}\left(  N\right)  $ such
that
\[
\mathcal{P}\left(  \mathcal{K}\right)  =\int_{\mathbf{I}}^{\oplus}
L^{2}\left(\mathbb{R}^{d}\right)  \otimes\mathbb{C}^{m\left(  \lambda\right)  }d\mu\left(  \lambda\right).
\]
If $\phi \in\mathcal{K}$ is a continuous wavelet such that $\{L(\gamma_{a,q,b}) \phi\}$ forms a Parseval frame, then $
m\left(  \lambda\right)  \leq\frac{1}{\mathbf{a}^{2}\left\vert \det B\left(
\lambda\right)  \right\vert }\text{ a.e. and }\left\Vert \phi\right\Vert ^{2}
\leq\int_{\mathbf{I}}\left(  b_{1}\cdots b_{d}\text{ }q_{1}\cdots q_{d}\right)
d\lambda.$
\end{proposition}

\begin{proof}
Assume there exists a function $\phi$ which is a continuous wavelet such that
$\left\{  L\left(  \gamma_{a,q,b}\right)  \phi\right\}  _{\gamma_{a,q,b}}$
forms a Parseval frame. The system $$\left\{  \mathbf{a}\left\vert \det B\left(  \lambda\right)  \right\vert
^{1/2}\pi_{\lambda}\left(  \gamma_{q,b}\right)  \widehat{\phi}\left(
\lambda\right)  \right\}  _{\gamma_{q,b}}$$ forms a Parseval frame for a.e $\lambda\in\mathbf{I}$ for the space
$L^{2}\left(\mathbb{R}^{d}\right)  \otimes\mathbb{C}^{m\left(  \lambda\right)  }.$ Thus, we have $\left\Vert \mathbf{a}\left\vert \det
B\left(  \lambda\right)  \right\vert ^{1/2}\widehat{\phi}\left(
\lambda\right)  \right\Vert ^{2}\leq1$, and
$$\left\Vert \widehat{\phi}\left(  \lambda\right)  \right\Vert ^{2}=m\left(
\lambda\right)  \leq\frac{1}{\mathbf{a}^{2}\left\vert \det B\left(
\lambda\right)  \right\vert }$$ by the admissibility of $\phi$ and Lemma \ref{ONB}. By the density condition of Gabor superframes (see Proposition 2.6 in \cite{Grog}), $\left\vert \det B\left(
\lambda\right)  \det A\left(  b\right)  \det D\left(  q\right)  \right\vert
\leq\frac{1}{m\left(  \lambda\right)  }$ a.e. Furthermore, because $\phi$ is a
continuous wavelet
\[
\left\Vert \widehat{\phi}\left(  \lambda\right)  \right\Vert ^{2}=m\left(
\lambda\right)  \leq\frac{1}{\left\vert \det B\left(  \lambda\right)  \det
A\left(  b\right)  \det D\left(  q\right)  \right\vert }.
\]
As a result,
\begin{align*}
\left\Vert \phi\right\Vert ^{2}  & =\int_{\mathbf{I}}\left\Vert \widehat{\phi
}\left(  \lambda\right)  \right\Vert ^{2}\left\vert \det B\left(
\lambda\right)  \right\vert d\lambda\\
& =\int_{\mathbf{I}}m\left(  \lambda\right)  \left\vert \det B\left(
\lambda\right)  \right\vert d\lambda\\
& \leq\int_{\mathbf{I}}\frac{d\lambda}{\left\vert \det A\left(  b\right)  \det
D\left(  q\right)  \right\vert }\\
& =\int_{\mathbf{I}}\left(  b_{1}\cdots b_{d}q_{1}\cdots q_{d}\right)
d\lambda.
\end{align*}
\end{proof}

\begin{theorem}\label{disc}
Let $\mathcal{H}$ be a multiplicity-free band-limited subspace of
$L^{2}\left(  N\right)  $ such that $\mathcal{P}\left(  \mathcal{H}\right)
=\int_{\mathbf{S}}^{\oplus}\left(  L^{2}\left(\mathbb{R}^{d}\right)  \otimes\mathbf{u}\right)  $ $d\mu\left(  \lambda\right)  $ and
\[
\mathbf{S}=\left\{  \lambda\in\mathbf{I} :\frac{\left\vert \det B\left(
\lambda\right)  \right\vert }{b_{1}\cdots b_{d}q_{1}\cdots q_{d}}%
\leq1\right\}
\]
with the following additional restriction on the quasi-lattice $\Gamma
_{a,q,b},$  $$b_{1}\cdots b_{d}q_{1}\cdots q_{d}a_{1}\cdots a_{n-2d}=1.$$ 
$\mathcal{H}$ admits a continuous wavelet $\phi$ which is discretizable by
$\Gamma_{a,q,b}$ in the sense that the operator $D_{\phi}%
:\mathcal{H\rightarrow} \:l^{2}\left(  \Gamma_{a,q,b}\right)  $ defined by
\[
D_{\phi}\psi\left(  \gamma_{a,q,b}\right)  =\left\langle \psi,L\left(
\gamma_{a,q,b}\right)  \phi\right\rangle
\]
is an isometric embedding of $\mathcal{H}$ into $l^{2}\left(  \Gamma
_{a,q,b}\right)  .$ Additionally, the discretized continuous wavelet generates
an orthonormal basis if $\mu\left(  \mathbf{S}\right)  =1.$
\end{theorem}
\begin{proof}
First, we start by defining a function $\phi$ such that $\mathcal{P}\left(
\phi\right)  \left(  \lambda\right)  =u_{\phi}\left(  \lambda\right)
\otimes\mathbf{u}$ for almost every $\lambda\in\mathbf{S}$. If we want to
construct $\phi$ such that $L\left(  \gamma_{a,q,b}\right)  \phi$ is a
Parseval frame for $\mathcal{H}$, it suffices to pick $u_{\phi}\left(
\lambda\right)  $ such that for a.e. $\lambda\in\mathbf{S},$%
\[
u_{\phi}\left(  \lambda\right)  =\frac{g\left(  \lambda\right)  }{\left(
a_{1}\cdots a_{n-2d}\left\vert \det B\left(  \lambda\right)  \right\vert
\right)  ^{1/2}}%
\]
and the Gabor system 
$
\mathcal{G}\left(  g\left(  \lambda\right)  ,A\left(  b\right)\mathbb{Z}^{d}\times B\left(  \lambda\right)  D\left(  q\right)\mathbb{Z}^{d}\right)$
generates a Parseval frame in $L^{2}\left(\mathbb{R}^{d}\right)  $. Since $$\dfrac{\left\vert \det B\left(  \lambda\right)
\right\vert }{b_{1}\cdots b_{d}q_{1}\cdots q_{d}}\leq1,$$ the density
condition is met almost everywhere and the existence of the measurable field
of functions $g\left(  \lambda\right)  $ generating Parseval frames is guaranteed by Lemma \ref{density}. To ensure that $\phi$ is a continuous wavelet, then we need to
check that for almost $\lambda\in\mathbf{S,}$ $\left\Vert u_{\phi}\left(
\lambda\right)  \right\Vert ^{2}=1.$ With some elementary computations, we
have%
\begin{align*}
\left\Vert u_{\phi}\left(  \lambda\right)  \right\Vert ^{2}  & =\frac
{\left\vert \det B\left(  \lambda\right)  \right\vert }{b_{1}\cdots b_{d}%
q_{1}\cdots q_{d}a_{1}\cdots a_{n-2d}\left\vert \det B\left(  \lambda\right)
\right\vert }\\
& =\frac{1}{b_{1}\cdots b_{d}q_{1}\cdots q_{d}a_{1}\cdots a_{n-2d}}\\
& =1.
\end{align*}
Finally, if $\phi$ is an orthonormal basis, then $
\left\Vert \phi\right\Vert ^{2}=\mu\left(  \mathbf{S}\right)  =1.$  This completes the proof. 
\end{proof}

\section{Examples}
\begin{example}
We consider the Heisenberg group realized as $N=P\rtimes M$ where $P=\exp\mathbb{R}Z\exp\mathbb{R}Y$ and $M=\exp\mathbb{R}
X$ with the following non-trivial Lie brackets: $\left[  X,Y\right]  =Z.$\end{example}
We have $\mathcal{P}\left(  L^{2}\left(  N\right)  \right)  =\int_{\mathbb{R}^{\ast}}^{\oplus}L^{2}\left(\mathbb{R}\right)  \otimes L^{2}\left(\mathbb{R}
\right)  \left\vert \lambda\right\vert d\lambda$. Consider for nonzero positive real
numbers $a,q,b$ the quasi-lattice $$\Gamma_{a,q,b}=\exp\left(  \frac{1}{a}\mathbb{Z}\right)  Z\exp\left(  \frac{1}{q}\mathbb{Z}\right)  Y\exp\left(  \frac{1}{b}\mathbb{Z}\right)X,$$ and the reduced quasi-lattice $\Gamma_{q,b}=\exp\left(  \frac
{1}{q}\mathbb{Z}\right)  Y\exp\left(  \frac{1}{b}\mathbb{Z}\right)  X.$ Let $$\mathcal{H}\left(  a\right)  =\mathcal{P}^{-1}\left(
\int_{\left( 0,a\right]}^{\oplus}L^{2}\left(\mathbb{R}\right)  \otimes\chi_{\left(  0,1\right]  }  \left\vert\lambda\right\vert d\lambda\right)  $$ be a left-invariant multiplicity-free
subspace of $L^{2}\left(  N\right)  $. Now put $b=a$ and choose $q$ such that
$1/q\leq1.$ By the density condition, there exists for each $\lambda
\in\left( 0,a\right]$ a function $g\left(  \lambda\right)  $ such that the Gabor system $\mathcal{G}\left(  g\left(  \lambda\right)  ,\frac{1}{a}\mathbb{Z}\times\frac{\left\vert \lambda\right\vert }{q}\mathbb{Z}\right)  $ forms a Parseval frame. For each $\lambda$ fix such function $g(\lambda),$ and let $\eta\in\mathcal{H}
\left(  a\right)  $ such that $$\left(  \mathcal{P}\eta\right)  \left(
\lambda\right)  =\frac{g\left(  \lambda\right)  }{\sqrt{a\left\vert
\lambda\right\vert }}\otimes\chi_{\left[  0,1\right]  }.$$
It follows that as long as $q$ is chosen such that $1/q\leq1,$ $L\left(
\Gamma_{a,q,a}\right)  \eta$ forms a Parseval frame for $\mathcal{H}(a)$. If we want to form an orthonormal basis generated by $\eta$, according to (\ref{basis}) we will need to pick $q$ such that $q=1/2$. However, this gives a contradiction, since $1/q=2>1$. Thus, there is no orthonormal basis of the form  $L\left(
\Gamma_{a,q,a}\right)  \eta$.
\begin{example} Let $N$ be a nilpotent Lie group with Lie algebra spanned by the basis
$\left\{  Z_{1},Z_{2},Y_{1},Y_{2},X_{1},X_{2}\right\}  $ with the following
non-trivial Lie brackets
$\left[  X_{1},Y_{1}\right] =Z_{1}$, $\left[  X_{2},Y_{2}\right]  =Z_{1}$, 
$\left[  X_{1},Y_{2}\right] =\left[  X_{2},Y_{1}\right]  =Z_{2}.$\end{example}
Let $\mathcal{H}$ be a left-invariant closed subspace of $L^{2}\left(  N\right)$, $$\mathbf{I}=\{(\lambda_1,\lambda_2,0,\cdots,0)\in\mathbb{R}^6:\vert \lambda_1^2-\lambda_2^2\vert\neq 0,0\leq \lambda_1\leq 2,0\leq \lambda_2 \leq 3 \},$$ with Plancherel measure $d\mu(\lambda_1,\lambda_2)=\vert \lambda_1^2-\lambda_2^2\vert d\lambda_1 d\lambda_2$ and
$$\mathcal{P}\left(  \mathcal{H}\right)  =\int_{\mathbf{I}}^{\oplus}\left(
L^{2}\left(\mathbb{R}^{2}\right)  \otimes\chi_{[0,1]^2}\right)  d\mu(\lambda_1,\lambda_2).$$ Since $\mathbf{s} =9,$ we
define the quasi-lattice,
\[
\Gamma_{\left(  2,3\right)  ,\left(  1,1\right)  ,\left(  3,3\right)
}= \exp\frac{\mathbb{Z}}{2}Z_{1}\exp\frac{\mathbb{Z}}{3}Z_{1}\exp
\mathbb{Z}Y_{1}\exp\mathbb{Z}Y_{2}\exp\frac{\mathbb{Z}}{3}X_{1}\exp\frac{\mathbb{Z}}{3}X_{2}  .
\]
Thus, there exists a function $\phi\in\mathcal{H}$ such that $L\left(
\Gamma_{\left(  2,3\right)  ,\left(  1,1\right)  ,\left(  3,3\right)
}\right)  \phi$ forms a Parseval frame. However, since $\mu\left(  \left[
0,2\right]  \times\left[  0,3\right]  \right)  =46/3\neq54,$ by (\ref{basis}) there is no orthonormal basis of the type $L\left(  \Gamma_{\left(
2,3\right)  ,\left(  1,1\right)  ,\left(  3,3\right)  }\right)  \phi$. In fact the norm of the vector $\phi$ can be computed to be precisely $(23/81)^{1/2}$. Since the multiplicity condition in Proposition \ref{corr} fails in this situation, there is no continuous wavelet which is discretizable by the lattice $\Gamma_{\left(  2,3\right)  ,\left(  1,1\right)  ,\left(  3,3\right).
}$
\begin{example} Let $N$ be a $9$ dimensional nilpotent Lie group with Lie algebra spanned by the basis $\{Z_i,Y_j,Y_k\}_{1\leq i,j,k\leq3}$ with the following non-trivial Lie brackets. $[Y_1,X_1]=[Y_3,X_2]=[Y_2,X_3]=Z_1$, $[Y_2,X_1]=[Y_1,X_2]=[Y_3,X_3]=Z_2$, and $[Y_3,X_1]=[Y_2,X_2]=[Y_1,X_3]=Z_3.$ \end{example} The Plancherel measure is $$ d\mu(\lambda_1,\lambda_2,\lambda_3)=|-\lambda_1^3-\lambda_2^3+\lambda_1\lambda_2\lambda_3-\lambda_3^3|d\lambda_1d\lambda_2d\lambda_3.$$ Assume that $\mathcal{H}$ is a multiplicity-free subspace of $L^2(N)$ with spectrum $\mathbf{S}=\{(\lambda_1,\lambda_2,\lambda_3,0,\cdots,0)\in\mathbb{R}^9: |-\lambda_1^3-\lambda_2^3+\lambda_1\lambda_2\lambda_3-\lambda_3^3|\leq 1, |-\lambda_1^3-\lambda_2^3+\lambda_1\lambda_2\lambda_3-\lambda_3^3|\neq 0\}\cap\mathbf{I},$  and $$\mathbf{I}=\{(\lambda_1,\lambda_2,\lambda_3,0,\cdots,0)\in\mathbb{R}^9:0\leq \lambda_i\leq 1,|-\lambda_1^3-\lambda_2^3+\lambda_1\lambda_2\lambda_3-\lambda_3^3|\neq 0 \}.$$ Put $a=b=q=(1,1,1).$ By Theorem \ref{disc}, the space $\mathcal{H}$ admits a continuous wavelet which is discretizable by $\Gamma_{a,b,q}$. 

\begin{center} \textbf{Acknowledgment}\end{center}
Sincerest thanks go to the referee for a very careful reading, and for many crucial, and helpful comments. His remarks were essential to the improvement of this paper.

\end{document}